\documentclass[12pt]{amsart}
\usepackage{amscd,amssymb,graphicx,color,a4wide,cite, amsmath}
\usepackage{listings}
\usepackage{color}
\usepackage{booktabs}

\usepackage{hyperref}
\usepackage{url}

\definecolor{dkgreen}{rgb}{0,0.6,0}
\definecolor{gray}{rgb}{0.5,0.5,0.5}
\definecolor{mauve}{rgb}{0.58,0,0.82}

\usepackage{mathrsfs}

\usepackage{epstopdf}

\usepackage[all]{xy}
\theoremstyle{theorem}
\newtheorem{thm}{Theorem}
\newtheorem{pro}[thm]{Proposition}

\newtheorem{eg}[thm]{Example}
\newtheorem{qu}[thm]{Question}

\theoremstyle{definition}
\newtheorem{dfn}[thm]{Definition}
\newtheorem{rem}[thm]{Remark}

\newcommand{\HH} {\mathbb{H}}
\newcommand{\MM} {\mathbb{M}}
\newcommand{\NN} {\mathbb{N}}
\newcommand{\ZZ} {\mathbb{Z}}
\newcommand{\QQ} {\mathbb{Q}}
\newcommand{\RR} {\mathbb{R}}
\newcommand{\CC} {\mathbb{C}}
\newcommand{\FF} {\mathbb{F}}
\newcommand{\PP} {\mathbb{P}}
\newcommand{\VV} {\mathbb{V}}
\renewcommand{\AA} {\mathbb{A}}
\newcommand{\GG} {\mathbb{G}}
\newcommand{\kk} {\mathbf{k}}

\def\mydate{\ifcase\month \or January\or February\or March\or
April\or May\or June\or July\or August\or September\or October\or 
November\or December\fi \space\number\day,\space\number\year}

\begin{document}

\title{The Birational Geometry of Ceva's Theorem}

\author{Thomas Prince}
\email{tmp@wincoll.ac.uk}

\date{\today}

\begin{abstract}
In this article we study Ceva's theorem and its higher-dimensional extensions from the perspective of algebraic and projective geometry. First, we situate the theorem within the study of algebraic surfaces by relating it to the defining equation of a del Pezzo surface of degree six inside the product of three projective lines. Second, by interpreting (higher-dimensional analogues of) Ceva's theorem in terms of projections from projective spaces, we recast these results as matrix completion problems. We use these ideas to offer proofs of some higher-dimensional analogues of Ceva's theorem. This article is written with a nonspecialist audience in mind and we hope that some useful context is provided in the form of remarks in the sections on surfaces for students of algebraic geometry.
\end{abstract}
\maketitle

\section{Introduction.}
We begin by recalling Ceva's theorem; see for example \cite[p.~4]{CoxGre67}. To do so, we fix a triangle $\triangle{ABC}$ and points $D$, $E$, $F$, situated on exactly one of the sides $BC$, $AC$, and $AB$ respectively.

\begin{figure}[h]
    \centering
    \includegraphics{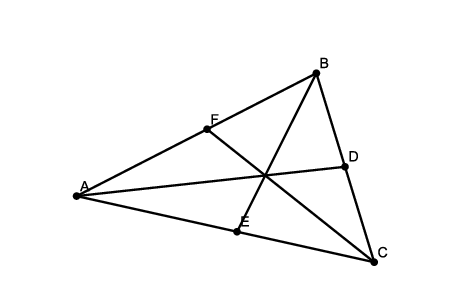}
    \caption{Ceva's theorem.}
\end{figure}

\begin{thm}[Ceva's theorem]
    The line segments $AD$, $BE$, and $CF$ coincide at a single point if and only if the following equality holds between ratios of side lengths on $\triangle{ABC}$:
    \[
    \frac{AE}{EC}\times \frac{CD}{DB} \times\frac{BF}{FA} = 1.
    \]
\end{thm}
If any of $D$, $E$, or $F$ is instead taken to lie on the line produced by $BC$, $AC$, or $AB$ respectively, the result still holds, although a signed length must be used.

\section{Ceva's Theorem from linear algebra.}
\label{sec:ceva1}
While Ceva's theorem does, of course, admit proofs within the framework of Euclidean geometry, we explain how to formulate the problem using coordinate geometry and linear algebra. This will allow us to fix many of the ideas used throughout the remainder of the article. As we observe below, this also provides a short proof of Ceva's theorem using determinants. The interpretation of Ceva's theorem as a vanishing determinant appears in work by Wernicke~\cite[p.~3]{Wern27}; our treatment also overlaps with that of Ben\'{i}tez~\cite{Benitez07}.

Crucially for us, the lengths that appear in Ceva's theorem appear only as ratios between collinear points. These ratios are invariant under the application of any injective linear map. Employing a standard technique from affine geometry, we view the plane containing $\triangle{ABC}$ as the ``height one slice" ($z=1$) of $\RR^3$, three-dimensional space. Fixing coordinates $(a_1,a_2)$, $(b_1,b_2)$, and $(c_1,c_2)$ for $A$, $B$, and $C$ respectively, we consider the transformation determined by the matrix
\[
T = \begin{pmatrix}
    a_1 & b_1 & c_1 \\
    a_2 & b_2 & c_2 \\
    1 & 1 & 1
\end{pmatrix}.
\]
The inverse of this transformation sends points with coordinates $(a_1,a_2,1)$, $(b_1,b_2,1)$, and $(c_1,c_2,1)$ to $(1,0,0)$, $(0,1,0)$, and $(0,0,1)$ respectively. The resulting triangle in three-dimensional space is illustrated in Figure~\ref{fig:cevain3d}. The points $D$, $E$, and $F$ map to points in the coordinate planes, and we let $(0,d_0,d_1)$, $(e_1,0,e_0)$, and $(f_0,f_1,0)$ denote the respective images of these three points. The point $(0,d_0,d_1)$ divides the segment between $(0,1,0)$ and $(0,0,1)$ in the ratio $d_1:d_0$ and hence $\frac{CD}{DB}$ becomes $d_0/d_1$. Similar expressions hold for the other two ratios: these become $e_0/e_1$ and $f_0/f_1$.

\begin{figure}[b]
    \centering
    \includegraphics{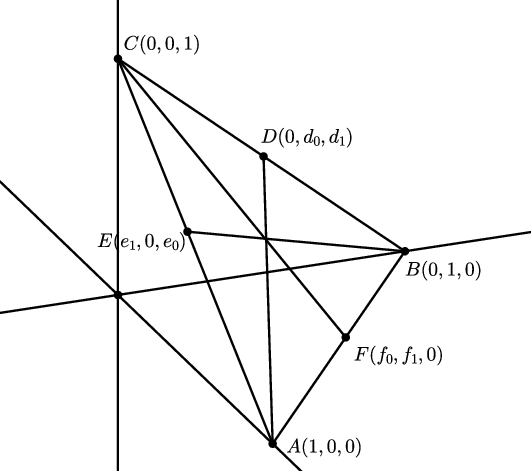}
    \caption{Embedding the triangle in $\RR^3$.}
    \label{fig:cevain3d}
\end{figure}

The question of whether $AD$, $BE$, and $CF$ are coincident can now be phrased as a question concerning the common intersection of three planes through the origin, spanned by $(1,0,0)$ and $(0,d_0,d_1)$, $(0,1,0)$ and $(e_1,0,e_0)$, or $(0,0,1)$ and $(f_0,f_1,0)$ respectively. This dictionary between lines in the plane and planes through the origin in $\RR^3$ will play a central role in what follows. To conclude this section however, we observe that Ceva's theorem follows by considering the vanishing of the determinant of a certain $3\times 3$ matrix.

\begin{proof}[Proof of Ceva's theorem]

The three planes considered above are those respectively annihilated by $\begin{pmatrix}0 & d_1 & -d_0 \end{pmatrix}$, $\begin{pmatrix}-e_0& 0& e_1\end{pmatrix}$, and $\begin{pmatrix}-f_1 & f_0 & 0\end{pmatrix}$, viewed as maps from $\RR^3$ to $\RR$. There is a line in $\RR^3$ simultaneously annihilated by all three of these maps if and only if the matrix
\[
\begin{pmatrix}
   0 & d_1 & -d_0 \\
   -e_0 & 0 & e_1 \\
   -f_1 & f_0 & 0 \\
\end{pmatrix}
\]
has a nontrivial kernel. This occurs exactly when the determinant of the above matrix vanishes, that is, when
\[
e_0(f_0d_0) - f_1(d_1e_1)=0.
\]
Rearranging this equation, we find that
\begin{equation}
\label{eq:Cevacoords}
\frac{d_0}{d_1} \times \frac{e_0}{e_1} \times \frac{f_0}{f_1}= 1.    
\end{equation}
This is nothing but the statement of Ceva's theorem in these coordinates.
\end{proof}

\section{Projective spaces.}

In this section and the next we recall some basic constructions of algebraic geometry. While we generally work over the real numbers, the natural setting for Ceva's theorem, we indicate where various generalizations are possible.

Given a vector space $V$ over a field $\kk$, we write $\PP(V)$ for the \emph{projective space} determined by $V$: the set of lines through the origin of $V$. In other words, $\PP(V)$ is the set of equivalence classes (excluding the singleton set containing the origin) of the relation $\sim$ for which $v \sim w$ if and only if the vectors $v$ and $w$ are proportional. Where the ground field $\kk$ is understood, we write $\PP^n$ for the projective space $\PP(\kk^{n+1})$. All vector spaces considered in this article are finite-dimensional.

The most important such space for us will be \emph{real projective plane}. As every line in $\RR^3$ through the origin can be written in the form
\[
\left\{ \lambda \begin{pmatrix}a\\b\\c\end{pmatrix} : \lambda \in \RR \right\},
\]
where $a$, $b$, and $c$ are not all equal to zero, we can specify any point in this projective space via its \emph{projective coordinates} $(a:b:c)$. By definition, we identify $(a:b:c)$ with $(\lambda a:\lambda b:\lambda c)$ for any $\lambda \neq 0$.

\begin{rem}
Despite describing the projective plane above as a ``set," it can --- and indeed should --- be considered as a set endowed with one of a range of possible additional structures. For example, since each line passing through the origin in $\RR^3$ meets the sphere of unit radius in precisely two antipodal points, each point in the projective plane corresponds to a pair of opposite points on the sphere. This description allows us to describe the real projective plane as a smooth manifold whose universal cover is a sphere.

There are also algebro-geometric structures one can consider, and indeed there is nothing to restrict our treatment to the field of real numbers. While the above topological description will cease to apply, we can instead describe projective spaces as \emph{varieties} or \emph{schemes}. We refer to texts by Hartshorne~\cite{Ha77} or Shafarevich~\cite{Sh94}, and to the introductory text of Smith \emph{et al.}~\cite{SmKaLaKe00}, for further detail on these structures.
\end{rem}

Reviewing the proof of Ceva's theorem given in Section~\ref{sec:ceva1}, we note that viewing the triangle $\triangle{ABC}$ as a subset of the plane $z=1$ in $\RR^3$ allows us to view it as a triangle in $\PP^2$. The vertices $A$, $B$, and $C$ are now lines containing the origin and the corresponding vertices (with $z$-coordinate equal to one). The transformation $T$ of the previous section identifies $A$, $B$, and $C$ with the points $(1,0,0)$, $(0,1,0)$, and $(0,0,1)$ respectively. It follows that this transformation identifies $A$, $B$, and $C$ with the points $(1:0:0)$, $(0:1:0)$, and $(0:0:1)$ in a projective plane with coordinates $(x_0:x_1:x_2)$.

Lines in the plane become \emph{projective lines}: the set of lines through the origin that are contained in a given plane in $\RR^3$. The edge $AB$, for example, is identified with the segment of the projective line $(a:b:0)$ where $a,b \geq 0$. We refer to the lines $x_0=0$, $x_1=0$, and $x_2=0$ as coordinate lines in $\PP^2$. Note that there are three of these, corresponding to the sets of lines contained in each of the coordinate planes of $\RR^3$. These three lines enclose the triangle with vertices $(1:0:0)$, $(0:1:0)$, and $(0:0:1)$.

In all that follows we rely on two key observations:
\begin{enumerate}
    \item Edges of a given triangle $\triangle{ABC}$ can be identified with (segments of) the three coordinate lines in $\PP^2$.
    \item In this language, the point $F$, for example, has coordinates $(f_0:f_1:0)$ in the \emph{projective coordinate line} $x_2=0$. The ratio $\frac{BF}{FA}$ is equal to $f_0/f_1$ and is well-defined as long as $f_1\neq 0$; that is, unless $F$ coincides with $A$.
\end{enumerate}

\section{Blowing up.}

The operation of \emph{blowing up} a surface, such as $\PP^2$, is a fundamental one, both to algebraic geometry in general and to our treatment of Ceva's theorem in particular. We begin by describing the blowup of a point in $\RR^2$.

Recalling that the points of $\PP^1$ are lines through the origin of $\RR^2$, we consider the set $S$ of pairs $(P,l)$ in $\RR^2 \times \PP^1$ that consists of points $P = (x,y) \in \RR^2$ and a line $l \in \PP^1$ through $O = (0,0)$ and $P$.

If $P$ is not equal to $O$, there is a unique line through $O$ and $P$ and so the only possible pair $(P,l)$ is given by $((x,y) , (x:y))$. If however $P = O$, then $l$ may be equal to any line through the origin. In fact, a pair $((x,y),(a_0:a_1))$ is contained in $S$ precisely when $xa_1=a_0y$. Note that, as expected, if $x$ and $y$ are not both zero, $a_0$, and $a_1$ are determined up to the usual overall scale factor. If however $x=y=0$, then any values of $a_0$ and $a_1$ are permitted.

\begin{dfn}
\label{dfn:blowup}
We refer to the solution set of the equation 
\begin{equation}
\label{eq:blowup}
xa_1 = a_0y
\end{equation}
in
$\RR^2 \times \PP^1$ as the \emph{blowup} of $\RR^2$ at the origin.
\end{dfn}

To give a topological description of this operation (only applicable over $\RR$): we have removed a small disc around $O$ from the plane and glued the boundary edge of a M\"obius band to the edge that removing this disc has created. The central circle of this M\"obius band is the projective line $\PP^1$ with coordinates $((0,0),(a_0:a_1))$. This projective line, inserted by blowing up, is referred to as the \emph{exceptional locus} of the blowup. To illustrate this blowup, we consider the locus $a_0\neq 0$, noting that we can write the coordinates of any point in this locus in the form $((x,y),(1:a_1))$. The defining equation of the blowup becomes $y=a_1x$, and this surface is illustrated in Figure~\ref{fig:blowup}. The vertical ($a_1$) axis is the exceptional locus while, for $x\neq 0$, the surface is the graph of the function $y/x$, the slope of the line between the origin and $(x,y)$. Figure~\ref{fig:blowup} also illustrates the way in which lines in the plane that pass through the origin are ``separated" by the blowup into disjoint lines that intersect the exceptional locus.

\begin{figure}
\centering
\includegraphics{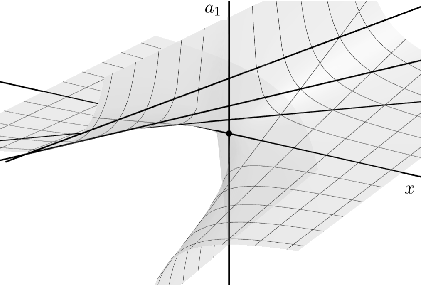} 
\caption{Blowing up the plane.}
\label{fig:blowup}
\end{figure}

Our description of the blowup as an incidence locus extends to define the blowup of the projective plane in a point. Indeed, fixing a point $O \in \PP^2$, we consider the set of pairs $(P,l)$ in which $l$ is a projective line passing through $O$ and $P$. Similarly to the affine case, and taking $O = (0:0:1)$, this locus can be described algebraically as
\[
\{((x_0:x_1:x_2),(y_0:y_1)) : x_0y_1=x_1y_0\}.
\]
Observe that if $x_2 \neq 0$ we can set $x_2=1$ and recover our previous description of a blowup.

\section{del Pezzo surfaces of degree six.}

In the Enriques--Kodaira classification of (complex) algebraic surfaces, an important collection of surfaces, called the \emph{del Pezzo} surfaces, is obtained by repeatedly blowing up the projective plane. Our description of Ceva's theorem relies particularly on one of these---the del Pezzo surface of degree six---obtained by blowing up the plane in three noncollinear points. Over the complex numbers, there are ten families of del Pezzo surfaces in total, including the projective plane, the product of two projective lines, and any cubic surface in $\PP^3$. Over the real numbers the situation is more subtle; see, for example, work of Koll\'ar~\cite{Koll01}. Once the projective plane is blown up in nine general points its geometry changes significantly, and the resulting surface becomes \emph{rational elliptic}.

Consider, as above, a triangle $\triangle{ABC}$ together with points $D$ on $BC$, $E$ on $AC$, and $F$ on $AB$. As in Section~\ref{sec:ceva1}, we may identify this triangle with the triangle in $\PP^2$ with vertices $(1 : 0 : 0)$, $(0: 1 : 0)$, and $(0: 0: 1)$. The points $D$, $E$, and $F$ have coordinates of the form $(0:d_0:d_1)$, $(e_1:0:e_0)$, and $(f_0:f_1:0)$ respectively. We also identify the projective line with coordinates $(d_0:d_1)$, for example, with lines in $\PP^2$ passing though $(1:0:0)$ via the identification of $(0:d_0:d_1)$ with the line containing $(1:0:0)$ and $(0:d_0:d_1)$.

Having fixed this notation, we construct the blowup of $\PP^2$ in the three points $A$, $B$, and $C$ as an incidence locus. This is given by the set of tuples $(P,l_A,l_B,l_C)$ consisting of a point $P$ in $\PP^2$, and lines $l_Q$ passing through $P$ and $Q$, where $Q$ is a point equal to one of $A$, $B$, or $C$. Recalling that these three points are identified with $(1:0:0)$, $(0:1:0)$, and $(0:0:1)$, the surface $S$ obtained via this blowup can be described algebraically as a subset of $\PP^2 \times \PP^1\times \PP^1\times \PP^1$ using the following equations:
\begin{align}
\label{eq:dP6eqns}
   x_1d_1 &= x_2d_0 \\ \nonumber
   x_2e_1 &= x_0e_0 \\ \nonumber
   x_0f_1 &= x_1f_0.  \nonumber
\end{align}

\begin{rem}
Our description of the blowup of a plane may seem a little ad hoc. The above description of $S$ can be reconciled with a more general theory of blowups in various ways. Perhaps the simplest makes use of the fact that blowing up does not affect (more precisely, the blowup map is an isomorphism on) any open set that does not contain the locus being blown up. Since we are blowing up the plane in a finite set of distinct points we can simply iterate the construction given by Definition~\ref{dfn:blowup} to obtain \eqref{eq:dP6eqns}.
\end{rem}

\section{Ceva's theorem revisited.}
\label{sec:ceva-again}
Ceva's theorem concerns the following question: \emph{given lines $l_A$, $l_B$, and $l_C$ through $A$, $B$, and $C$ respectively, when are these lines coincident at some point $P$?} In the previous section we constructed the collection $S$ of all such coincident lines, together with the point $P$ at which they coincide. This surface is the subset
\[
S \subset \PP^2 \times \PP^1 \times \PP^1 \times \PP^1
\]
defined by \eqref{eq:dP6eqns}. For brevity, we denote the product of $n$ copies of a projective space $\PP^k$ by $(\PP^k)^n$ in what follows. Consider now the image of $S$ after applying the projection
\[
\pi \colon \PP^2 \times (\PP^1)^3 \to (\PP^1)^3.
\]
This projection ``forgets" the point $P$ and its image is precisely the set of triples of lines $(l_A,l_B,l_C)$ that coincide at a point.

\begin{pro}
\label{pro:ceva2}
The image of $S$ in $(\PP^1)^3$ has equation $d_0e_0f_0=d_1e_1f_1$.
\end{pro}
\begin{proof}
We will in fact show directly that $\pi$ is an isomorphism (an invertible algebraic map) of varieties from $S$ to the hypersurface $H$ defined by $d_0e_0f_0=d_1e_1f_1$. We first note that, studying the equations \eqref{eq:dP6eqns}, neither the point $((1:0),(1:0),(1:0))$, nor $((0:1),(0:1),(0:1))$, is an element of $S$. We therefore have that, for any point of $S$, either two of $d_0, e_0, f_0$ are nonzero or two of $d_1, e_1, f_1$ are. Exploiting the symmetries of the equations \eqref{eq:dP6eqns}, we may assume without loss of generality that $d_1$, $e_0$, and $f_0$ are all nonzero.

Observe that, as $d_1$, $e_0$, and $f_0$ are nonzero, if $x_2=0$ we must have that $x_0$ and $x_1$ both vanish, which is not permitted. Hence $x_2 \neq 0$ and we can set $x_2=1$. This reduces the equations in \eqref{eq:dP6eqns} to
\begin{align*}
   x_1d_1 &= d_0 \\
   e_1 &= x_0e_0 \\
   x_0f_1 &= x_1f_0, \\
\end{align*}
from which $x_0=e_1/e_0$, $x_1 = d_0/d_1$. Noting that $x_0f_1 = x_1f_0$ we conclude that $d_0e_0f_0=d_1e_1f_1$, as required. 
\end{proof}

\begin{rem}
If we restrict to points for which $x_0x_1x_2$ is nonzero, after multiplying the equations in \eqref{eq:dP6eqns} together and dividing by $x_0x_1x_2$, we find that the image of $S$ contains a (Zariski) open set of $H$. It is possible at this point to argue from general results in algebraic geometry that the image of $S$ is the closure of this open set. As $H$ is irreducible, this closure must be $H$ itself. The argument is that, since $S$ is projective and $H$ is separated, the map from $S$ to $H$ is proper \cite[Proposition~$12.58(3)$]{GoWe20}, and hence its image is closed. This is analogous to the result in general topology that the image of map from a compact space to a Hausdorff one is closed.
\end{rem}

Proposition~\ref{pro:ceva2} thus provides an equation that determines when three lines, originating from fixed points, are coincident; that is, when 
\[
\frac{d_0}{d_1} \times \frac{e_0}{e_1} \times \frac{f_0}{f_1}=1,
\]
which is Ceva's theorem.

\begin{rem}
    Each of the monomials $d_0e_0f_0$ and $d_1e_1f_1$ can be viewed as sections of a \emph{line bundle} over $(\PP^1)^3$ of tridegree $(1,1,1)$. $S$, the blowup of $\PP^2$ in three (non-collinear) points is thereby demonstrated to be the vanishing locus of a global section of this line bundle. This is a standard description of the del Pezzo surface of degree six.
\end{rem}

\section{A birational view of Ceva's theorem.}
\label{sec:birational-ceva}
So far our constructions have been \emph{projective}: we have obtained an isomorphism of projective varieties from a del Pezzo surface $X$ of degree six to a hypersurface in $(\PP^1)^3$. In this and subsequent sections, we describe the restriction of this hypersurface to certain open sets, allowing us to make use of (bi)rational maps: algebraically defined maps that may not be defined on the entirety of the algebraic variety in question.

In what follows, a key role will be played by \emph{projections} from projective spaces. Given a point $P$ in $\PP^2$ with coordinates $(x_0:x_1:x_2)$, we can project from, for example, the point $(1:0:0)$, to obtain the point $(x_1:x_2)$ in $\PP^1$. Geometrically, this is the point in $\PP^1$ obtained by taking the point of intersection between $x_0=0$ and the line passing through both $(1:0:0)$ and $P$. Note that this projection is only well-defined if $P$ is in the complement of $(1:0:0)$. Alternatively, viewing $\PP^1$ as the space of lines through $(1:0:0)$, the point $(x_1:x_2)$ can be interpreted simply as the line passing through $P$ and $(1:0:0)$ itself.

If we restrict to the open set in $\PP^2$ consisting of points $(x_0:x_1:x_2)$ such that no two of $x_0$, $x_1$, and $x_2$ are simultaneously zero, we can define the map
\[
\PP^2 \dashrightarrow (\PP^1)^3
\]
sending $(x_0:x_1:x_2)$ to $((x_1:x_2),(x_2:x_0),(x_0:x_1))$. This replaces a point in $\PP^2$ with each of its three projections to coordinate lines. The dashed arrow indicated that this is a rational map, that is, defined on an open set in $\PP^2$. On this open set, this map coincides with the composition
\[
\PP^2 \dashrightarrow \PP^2 \times(\PP^1)^3 \to (\PP^1)^3.
\]
Here the left-hand (rational) map is given by sending $P$ to $(P,l_A,l_B,l_C)$, where $l_Q$ is the unique line passing through $P$ and $Q$ (as $P$ is distinct from $A$, $B$, and $C$); the right-hand map is given by the projection onto $(\PP^1)^3$.

Using the above, we can rephrase the question answered by Ceva's theorem as follows.

\begin{qu}
\label{qu:ceva-recast}
    Given a triple $((d_0:d_1),(e_0:e_1),(f_0:f_1)) \in (\PP^1)^3$, when is this triple given by the three projections of a point $(x_0:x_1:x_2)$ to coordinate lines in $\PP^2$?
\end{qu}

Reformulating this question yet again, we can consider the matrix
\begin{align}
\label{eq:matrix-completion}
\begin{pmatrix}
    \star & d_0 & d_1 \\ 
    e_1 & \star & e_0 \\ 
    f_0 & f_1 & \star
\end{pmatrix}.
\end{align}

Question~\ref{qu:ceva-recast} is answered affirmatively precisely when we can complete this matrix to one for which any pair of rows differ only by an overall scale; in other words, when we can complete this to a matrix of rank one. This is an example of a (low rank) \emph{matrix completion problem}, generalizations of which have important applications in statistics, machine learning, and algebraic matroid theory, among other fields. We refer the interested reader to the survey works \cite{Jo90, LiHuChZh19} and to \cite{ChiLi19,FanCh17,RoSiTh20} for a small sample of the literature on matrix completion relevant to the fields listed above.

\section{Higher dimensions: preliminaries.}

In this section and the next we replace the triangle $\triangle{ABC}$ with an $n$-dimensional simplex $\Delta$ which we fix for the remainder of this article. As in the two-dimensional case, we can identify the vertices of $\Delta$ with the points $(1:0:\cdots:0), \ldots, (0:\cdots:0:1)$ in $\PP^n$. All the faces of $\Delta$ are contained in coordinate subspaces of this projective space, that is, each face is determined by the vanishing of a subset of the variables $x_0,\ldots, x_n$ on $\PP^n$. To describe such coordinate subspaces, we make use of the following notation. 

\begin{dfn}
    Given a subset $I = \{i_0,\ldots,i_k\}$ of $\{0,\ldots,n\}$ we let both $\langle i_0,\ldots, i_k \rangle$ and $\langle I \rangle$ denote the (projective) linear space in $\PP^n$ formed by the vanishing of all coordinate functions $x_j$ for $j \notin I$.
\end{dfn}

The projections from a point in $\PP^2$ described in the previous section generalize naturally to (rational) maps from any projective space. To illustrate this, fix a positive integer $n$ and a subset $I \subset \{0,\ldots,n\}$ of size $k+1$, where $k<n$. Writing $I = \{i_0,\ldots, i_k\}$, there is a projection
\[
\pi_I \colon \PP^n \dashrightarrow \PP^k
\]
which maps $(x_0:\cdots : x_n) \mapsto (x_{i_0}:\cdots: x_{i_k})$. In the notation introduced above, we can identify the image of this projection with $\langle I \rangle$. To describe this projection geometrically we introduce two pieces of terminology.

\begin{dfn}
Fixing a $k$-dimensional face $F$ of $\Delta$, there is a face $F'$ of $\Delta$ \emph{opposite} $F$: the smallest face of $\Delta$ containing those vertices not contained in $F$.    
\end{dfn}

Next, given its fundamental role in what follows, we explicitly formulate what we mean by the (linear) span of a face and a point of $\Delta$, viewed as subsets of $\PP^n$.

\begin{dfn}
    Given a collection of points $S$ in $\PP^n$ the \emph{span} of $S$ is the smallest linear subspace of $\PP^n$ that contains $S$. Given a face $F$ of $\Delta$ and a point $P \notin F$ the span of $F$ and $P$ is the span of the union of $\{P\}$ with the set of vertices of $F$.
\end{dfn}

The role of the \emph{cevians}, the lines containing a vertex of a triangle and a point on the opposite edge, is played by the span of a point in a face $F$ of $\Delta$ and the face opposite $F$. Moreover, the projection of a point $P$ in $\Delta$ from $F$ is the unique point in the face $F'$ opposite $F$ given by the intersection of $F'$ with the span of $P$ and $F$.

\begin{eg}
    Consider the map sending $(x_0:x_1:x_2:x_3)$ to $(x_2:x_3)$. This projection is defined in the complement of the line $x_2=x_3=0$. Geometrically one finds the unique plane containing $(x_0:x_1:x_2:x_3)$ and the line $x_2=x_3=0$. This plane intersects the line $x_0=x_1=0$ in the single point $(0:0:x_2:x_3)$ which we can view as the image of this projection. See Figure~\ref{fig:plane-in-p3} for an illustration of this projection.
    \begin{figure}
        \centering
        \includegraphics{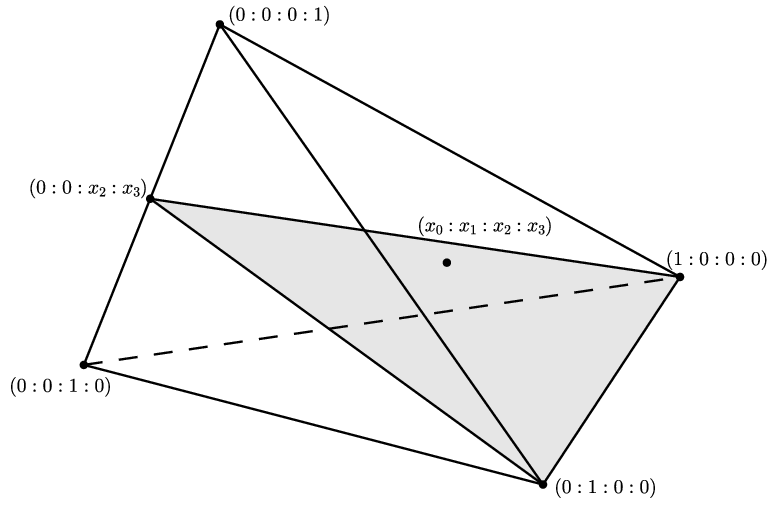}
        \caption{Projecting a point from a line in $\PP^3$.}
        \label{fig:plane-in-p3}
    \end{figure}
\end{eg}

\section{Extensions of Ceva's theorem.}
For fixed $n$ and $k$, the natural generalization of the question answered by Ceva's theorem---as formulated in Question~\ref{qu:ceva-recast}---becomes the following.

\begin{qu}
    \label{qu:ceva-generalization}
    Fix a point $P$ in $\langle I \rangle$ in each $(k+1)$-dimensional subset $I$ of $\{0,\ldots, n\}$. When is this collection of points given by the projection of a single point $P$ in $\PP^n$ to each $k$-dimensional coordinate space?
\end{qu}

Note that any $n$-dimensional simplex $\Delta$ can be identified with the intersection of the positive orthant in $\RR^{n+1}$ and the hyperplane in which the coordinates sum to one. Using this identification, a point $P$ in $\Delta$ can be given, using projective coordinates, as $(x_0:\cdots:x_n)$. Given a subset $I \subseteq \{0,\ldots,n\}$ of size $k+1$, let $F$ (unrelated to the point $F$ which appeared previously) denote the face of $\Delta$ contained in the linear space $\langle I\rangle$. We observe that the intersection of $F$ with the span of $P$ and the face opposite $F$ has coordinates $x_i$ for $i \in I$ and $0$ otherwise. Thus we can reinterpret Question~\ref{qu:ceva-generalization} in the following way.

\begin{qu}
    \label{qu:ceva-generalization-geo}
    Fix a point in every $k$-dimensional face $F$ of $\Delta$ and not contained in any $(k-1)$-dimensional face. For each $F$ we consider the span of the given point in $F$ and the face opposite $F$. When do these linear subspaces intersect in a point $P \in \Delta$?
\end{qu}

Ceva's theorem itself answers this question in the case $n=2$, $k=1$ by providing a generator of the ideal of functions vanishing on the algebraic set given by the image of the rational map $\PP^2 \dashrightarrow (\PP^1)^3$ in which each component is one of the maps $\pi_I$. In general, a choice of point in every $k$-dimensional face of $\Delta$ (or in the span of this face) can be described as a point in $\prod_I{\PP^k}$, where $I$ runs over all subsets of $\{0,\ldots, n\}$ of size $k$. To rephrase questions~\ref{qu:ceva-generalization} and \ref{qu:ceva-generalization-geo} in these terms, we restrict to the subset $U$ of $\prod_I\PP^k$ in which all coordinates are nonzero. This has the advantage that it describes an affine algebraic variety (rather than a quasi-projective one) and coincides with the usual setting of Ceva's theorem (as cevians are not generally permitted to intersect on an edge of the triangle). Let $V$ denote the intersection of the image of $\pi_I$ with $U$.

\begin{qu}
    \label{qu:ceva-alggeom}
    Describe an ideal $I$ whose vanishing locus $V(I)$ coincides with $V$ on the open set $U \subset \prod_I\PP^k$.
\end{qu}

\begin{rem}
    The set $U$ is the product of open sets of projective spaces $\PP^k$ in which every coordinate is nonzero. This set is acted on by the product ${\big(\kk^\star\big)}^k$ of $k$ copies of the group of (multiplicative) units in $\kk$. As such this is set often described as a \emph{torus} (although admittedly it bears little resemblance to one over the reals). The ideal $I(V)$ is an example of a \emph{toric} ideal. Maps and varieties admitting an action of a torus of this form are the subject of \emph{toric geometry}; see Fulton~\cite{Ful93}.
\end{rem}

Two cases of this extension of Ceva's theorem have already been examined in detail. First, the case in which $k=n-1$ has been studied by Landy~\cite{Landy88} and Samet~\cite{Samet21}. In this case an element in $\prod_I\PP^k$ corresponds to a choice of point in each (hyperplane containing a) facet, that is, an $(n-1)$-dimensional face of $\Delta$. The projective coordinates describing the point in $\PP^n$ at which the cevians---the lines from the vertices to the given point in the facet opposite---are interpreted in \cite{Landy88} as a mass distribution on the vertices of the simplex.

Second, and at the other extreme (although these ``extremes" coincide when $n=2$), one can consider the case $k=1$. This is the case treated by Witczy\'{n}ski~\cite{Wit95, Wit96} and Buba-Brzozowa~\cite{BuBr00} and corresponds to choosing a point in each of the coordinate lines of $\PP^n$ (one for each pair $i<j$ in $\{0,\ldots, n\}$). We let $P_{i,j}$ denote this point and let $(x^{i,j}_0:x^{i,j}_1)$ denote its coordinates. Identifying the projective line with these coordinates with $\langle i,j\rangle$ identifies $(x^{i,j}_0:x^{i,j}_1)$ with
\[
(0 : \cdots : x^{i,j}_0 : 0 : \cdots : 0 : x^{i,j}_1 : \cdots : 0),
\]
a point with nonzero entries in the $i$th and $j$th places.

\begin{thm}[cf. Theorem~$1$ of \cite{BuBr00}]
\label{thm:casek1}
	Given points $P_{i,j}$ in each projective coordinate line, as above, the hyperplanes containing $P_{i,j}$ and the face of $\Delta$ opposite the edge contained in $\langle i,j \rangle$ are coincident if and only if the coordinates $x^{i,j}_0$ and $x^{i,j}_1$ of the points $P_{i,j}$ are subject to the equations
    \begin{equation}
	\label{eq:casek1}
	\frac{x^{a,b}_0}{x^{a,b}_1}\times \frac{x^{b,c}_0}{x^{b,c}_1} \times \frac{x^{a,c}_1}{x^{a,c}_0} = 1
	\end{equation}
    for all $0 \leq a < b < c \leq n$.
\end{thm}
\begin{rem}
The ratios in \eqref{eq:casek1} are, of course, the same as those appearing in \cite[Proposition~$2$]{Wit96} and \cite[Theorem~$1$]{BuBr00}. The lack of continuity in the ordering of indices $1$ and $0$ in the final fraction is caused by the cyclic ordering of the indices $\{a,b,c\}$ in Ceva's theorem, but the compatibility of the definition of $x^{a,c}_0$ and $x^{a,c}_1$ with the ordering $a < b < c$.
\end{rem}
\begin{proof}
    Noting that formula \eqref{eq:casek1} is Ceva's theorem applied to the triangles formed by triples of coordinate lines, it suffices to prove that the points $P_{i,j}$ are obtained by projection from a single point in $\PP^n$ if and only if, for each triple $a<b<c$ in $\{0,\ldots,n\}$, each of $P_{a,b}$, $P_{a,c}$, and $P_{b,c}$ is obtained by projection from a point in the projective plane $\langle a,b,c \rangle \subset \PP^n$.

    We seek coordinates $x_0,\ldots,x_n$ such that $P_{i,j}$ is given by $(x_i:x_j)$ for all $i < j$. Begin by setting $x_0=1$. Writing $P_{0,j}$ as $(1:y_j)$ for all $0 < j < n$ fixes a point $P = (1:y_1:\cdots:y_n) \in \PP^n$. Consider the projection of $P$ to $\langle i,j \rangle$. Since the triple of points $(1:y_i)$, $(1:y_j)$, and $(x^{i,j}_0:x^{i,j}_1)$ satisfies Ceva's theorem, we have that $(x^{i,j}_0:x^{i,j}_1)$ is equal to the projection of $(1:y_i:y_j)$ from $(1:0:0)$. That is, $(x^{i,j}_0:x^{i,j}_1) = (y_i:y_j)$. Hence the points $P_{i,j}$ are obtained by projection from $P \in \PP^n$, as required.
\end{proof}

Rather than treating the case $k=n-1$ individually, we simultaneously treat all cases in which $k > 1$ via a matrix completion problem extending that described in Section~\ref{sec:birational-ceva}.

\begin{rem}
    While we obtain a relatively succinct formulation of these extensions of Ceva's theorem, in the case $k > 1$ our equations do not closely resemble the case $k=1$; moreover, our coordinates do not have as geometric an interpretation as that given by Samet in \cite{Samet21}. We can however, following \cite{Landy88}, interpret the homogeneous coordinates on each factor of $\PP^k$ as a mass distribution or weighting of the vertices which determines the specified point in each $k$-dimensional face of $\Delta$.
\end{rem}

To formulate our result, we choose an ordering of the $k+1$ subsets of $\{0,\ldots,n\}$. For each such subset $I$ and element $j \in I$, let $x^I_j$ denote the $j$th coordinate of a given point $P_I \in \langle I \rangle$. We also note that the linear space $\langle I \rangle$ is itself identified with the $I$th factor of $\prod_I\PP^k$. Let $M_{n,k}$ denote the $\begin{pmatrix} n+1 \\ k+1 \end{pmatrix}$ by $n+1$ matrix whose $(I,j)$th entry is equal to $x^I_j$ if $j \in I$ and is unspecified otherwise.

\begin{eg}
\label{eq:ceva-matrix2}
    The matrix $M_{2,1}$ is equal to
    \[
    \begin{pmatrix}
        \star & x^{1,2}_1 & x^{1,2}_2  \\
        x^{0,2}_0 & \star & x^{0,2}_2  \\
        x^{0,1}_0 & x^{0,1}_1 & \star  \\        
    \end{pmatrix}.
    \]
    After suitable relabelling of variables, this is precisely the matrix \eqref{eq:matrix-completion}.
\end{eg}

In terms of this matrix, Question~\ref{qu:ceva-generalization} asks for conditions on the (specified) entries that ensure that the matrix can be completed to one of rank one. Any row of this matrix will then provide homogeneous coordinates specifying the desired point of intersection.

\begin{pro}
    \label{pro:minors}
    If $k > 1$, $V$ is the vanishing locus of the $2\times 2$ minors of $M_{n,k}$ that do not contain any unspecified entries. 
\end{pro}
\begin{proof}
    As in the proof of Theorem~\ref{thm:casek1}, we first interpret the statement in terms of projections of $\PP^n$. For the points $P_I$ to be obtained by projection from a single point $P$, it is clearly necessary that the projections of $P_I$ and $P_J$ to the line $\langle i,j\rangle$ coincide whenever $\{i,j\} \subset I\cap J$. That is, the projections of points in $k$-dimensional subspaces to a common line must coincide. This condition is precisely the vanishing of the $2\times 2$ minor of $M_{n,k}$ 
    \[
    x^I_ix^J_j - x^I_jx^J_i = 0.
    \]
    Moreover, the vanishing of any such $2\times 2$ minor describes the coincidence of the projections of $P_I$ and $P_J$ to a shared line $\langle i,j\rangle$. Following the notation used in the proof of Theorem~\ref{thm:casek1}, denote the point obtained by projection to the line $\langle i,j\rangle$ by $P_{i,j}$. Given that $k\geq 2$, it must also be the case that for any triple $a<b<c$, the points $P_{a,b}$, $P_{a,c}$, and $P_{b,c}$ are obtained by projection from a point in $\langle a,b,c\rangle$. Indeed, take any $I$ such that $\{a,b,c\} \subseteq I$ and project $P_I$ to $\langle a,b,c\rangle$. This point must then project to each of $P_{a,b}$, $P_{a,c}$, and $P_{b,c}$. Hence the set of points $\{P_{i,j} : 0 \leq i < j \leq n \}$ satisfies the conditions of Theorem~\ref{thm:casek1} and these are hence obtained by projection from a point $P \in \PP^n$. Letting $P'_I$ denote the projection of $P$ to a linear space $\langle I \rangle$ of dimension $k$, we observe that $P'_I=P_I$ as these points have the same projection to every coordinate line.    
\end{proof}

We note, to conclude, that while Theorem~\ref{thm:casek1} and Proposition~\ref{pro:minors} describe ideals whose vanishing locus coincides with $V$, as asked by Question~\ref{qu:ceva-alggeom}, one could consider various refinements of this question. For example, we have not attempted to show that the ideals described are radical, nor have we attempted to describe ideals in which any coordinate is permitted to vanish, as we have in the case $n=2$, $k=1$. Certainly aspects of the geometric picture described in Section~\ref{sec:ceva-again} translate to higher dimensions, in which the natural object generalizing the blowup of the projective plane in three points is the blowup of $\PP^n$ in the union of its $(n-k)$-dimensional coordinate subspaces.

Finally, we offer a further extension of the questions considered above, suggested by the matrix completion approach used in Proposition~\ref{pro:minors}. Considering the matrix $M_{n,k}$ as defined above, we can ask for necessary and sufficient conditions for it to be possible to complete this matrix to one of rank $(r+1)$, for a specified value of $r \geq 0$. Reformulating this question geometrically leads us to the following:
\begin{qu}
Fixing points in each of the $k$-dimensional faces of an $n$-dimensional simplex $\Delta$, when does there exist an $r$-dimensional linear space that intersects all the linear spaces obtained as the span of each such point and its opposite face?
\end{qu}

\section*{Acknowledgments.}
The author wishes to thank the anonymous referees for their valuable comments and feedback on this article.

\bibliographystyle{plain}
\bibliography{biblio}

\end{document}